\documentclass[12pt,reqno]{amsart}

\usepackage{color}
\usepackage[linktocpage=true]{hyperref}
\usepackage{url}
\usepackage{amsthm}
\usepackage{amsmath}
\usepackage{amssymb}
\usepackage{amstext}
\usepackage{stmaryrd}
\usepackage{latexsym}
\usepackage{enumitem}
\usepackage{comment}
\usepackage[left=3.3cm, right=3.3cm, paperheight=13in]{geometry}

\newtheorem{theorem}{Theorem}[section]

\newtheorem{lemma}[theorem]{Lemma}

\theoremstyle{definition}

\newtheorem{question}{Question}

\newtheorem{example}{Example}
\newtheorem{problem}{Problem}
\newtheorem*{example*}{Example}

\hypersetup{
    pdftitle={On a system of equations with primes},
    pdfauthor={Paolo Leonetti and Salvatore Tringali},
    colorlinks=true,        
    linkcolor=blue,          
    citecolor=green,         
    urlcolor=cyan           
}

\newcommand\op{\mathrm{op}}

\begin{document}

\title{On a system of equations with primes}
\author{Paolo Leonetti}
\address{Universit\`a Bocconi, via Sarfatti 25, 20100 Milan, Italy.}
\email{leonetti.paolo@gmail.com}

\author{Salvatore Tringali}
\address{Laboratoire Jacques-Louis Lions (LJLL),
Universit\'e Pierre et Marie Curie (UPMC), 4 place Jussieu, bo\^ite courrier
187, 75252 Paris (cedex 05), France.}
\email{tringali@ann.jussieu.fr}
\urladdr{http://www.math.polytechnique.fr/~tringali/}

\begin{abstract}
\noindent{} Given an integer $n \ge 3$, let $u_1, \ldots, u_n$ be pairwise coprime integers $\ge 2$, $\mathcal D$ a family of nonempty proper subsets
of $\{1, \ldots, n\}$ with ``enough'' elements, and $\varepsilon$ a function $
\mathcal D
\to \{\pm 1\}$. Does there exist at least one prime
$q$ such that $q$ divides
$\prod_{i \in I}
u_i - \varepsilon(I)$ for some $I \in \mathcal D$, but it does not divide $u_1 \cdots u_n$? 

We
answer this question in the positive  when the $u_i$ are
prime powers and $\varepsilon$ and $\mathcal D$ are subjected to certain restrictions.
We use the result to prove that, if $\varepsilon_0 \in \{\pm 1\}$ and $A$ is a
set of three or more primes that contains all prime
divisors of any number of the form $\prod_{p \in B} p - \varepsilon_0$ for
which $B$ is a finite nonempty proper subset of $A$, then $A$ contains all the
primes.
\end{abstract}
\maketitle
\section{Introduction}

Let $\mathbb P := \{2, 3, \ldots\}$ be the set of all (positive
rational) primes. There
are several proofs of the fact that $\mathbb P$ is infinite: Some are
elementary, others come as a byproduct of deeper
results. E.g., six of them, including Euclid's classical proof, are given by M.
Aigner and
G. M. Ziegler in the first chapter of their lovely \textit{Proofs from THE
BOOK} \cite{AZ10}.
Although not really focused on the infinity of primes, this
paper is inspired by Euclid's original work on the subject, concerned as it is
with the factorization of numbers of the form $a_1 \cdots
a_n \pm 1$, where $a_1, \ldots, a_n$ are coprime positive integers, and
in fact prime powers (we do not consider $1$ as a prime power).
To be more precise, we first need to fix some notation.

We write $\mathbb Z$ for the integers, $\mathbb N$ for the nonnegative integers, and $\mathbb N^+$ for $\mathbb N \setminus \{0\}$, each of these sets being endowed with its usual addition $+$,
multiplication $\cdot$ and total order $\le$ (as
is customary, $\ge$ will stand for the dual order of $\le$).

For a
set $A$, we denote by $|A|$ the cardinality of
$A$, and by $\mathcal P_\star(A)$ the family of
all finite nonempty \textit{proper} subsets of $A$,
in such a way that $A \notin \mathcal P_\star(A)$. Furthermore, for an integer $n \ge 1$ we set $S_n:=\{1,
\ldots, n\}$ and let
$\mathcal P_n(A)$ be the collection of all subsets $B$ of $A$ with $|B| = n$.

For the notation and terminology used herein without definition, as well as for
material
concerning classical topics in number theory, the reader should refer to
\cite{Hardy08}.

With that said, we can state the basic question
addressed by the paper:
\begin{question}
\label{quest1}
Given an integer $n \ge 3$, pick exponents $v_1, \ldots, v_n \in
\mathbb N^+$ and (pairwise) distinct primes $p_1, \ldots, p_n \in \mathbb P$, and let $\mathcal D$ be a nonempty subfamily of
$\mathcal P_\star(S_n)$ with ``enough'' elements and $\varepsilon$ a map
$\mathcal P_\star(S_n)
\to \{\pm 1\}$. Does there exist at least one prime
$q \in \mathbb P \setminus \{p_1, \ldots, p_n\}$ such that $q$ divides
$\prod_{i \in I}
p_i^{v_i} - \varepsilon(I)$ for some $I \in \mathcal D$?
\end{question}
At present, we have no formal definition of what
should be meant by the word ``enough'' in the previous statement: this is
part of the question.

With the notation from above it is
rather clear, for instance, that the
answer to Question \ref{quest1} is no, at least in general, if $|\mathcal D|$ is ``small'' with respect to $n$, as shown by the following:
\begin{example}
Given an integer $k \ge 3$, distinct primes $q_1, \ldots, q_k$ and positive integers $e_1, \ldots, e_k$, let $q$ be the greatest prime dividing at least one of the numbers of the form $\prod_{i \in I} q_i^{e_i} \pm 1$ for $I \in \mathcal P_\star(S_k)$.

Then, we get a negative answer to Question \ref{quest1} by extending $q_1, \ldots, q_k$ to a sequence $q_1, \ldots, q_\ell$ containing all the primes $\le q$ (note that $\ell \ge k+1$), by taking a nonempty $\mathcal E \subseteq \mathcal P_\star(S_k)$ and arbitrary $e_{k+1}, \ldots, e_\ell \in \mathbb N^+$, and by setting $n := \ell$, $p_i := q_i$, $v_i := e_i$ and $\mathcal D := \mathcal E$.
\end{example}
Thus, to
rule
out such trivial
cases, one shall
suppose, e.g., that $|\mathcal D| \ge n \kappa$ or, in alternative, $|\mathcal
D| \ge n^{\kappa}$ for some
absolute constant $\kappa > 0$.

Specifically, we concentrate here on
the case where $\mathcal
D$ contains at least all subsets of $S_n$ of size $1$, $n-2$, or $n-1$, and the
restriction of  $\varepsilon$ to these subsets is constant (see Theorem
\ref{th:1} below), while collecting a series of intermediate results
that could be useful, in future research, to try to draw broader
conclusions.

We observe, in this sense, that Question \ref{quest1} can be
``generalized''
as follows:
\begin{question}
\label{quest2}
For an integer $n \ge 3$, let $u_1,
\ldots,
u_n$ be pairwise coprime integers $\ge 2$, $\mathcal D$ a
nonempty subcollection of $\mathcal P_\star(S_n)$ for which
$\mathcal D$ has ``enough'' elements, and $\varepsilon$ a function $\mathcal P_\star(S_n)
\to \{\pm 1\}$. Does there exist at least one
prime $q$ such that $q$ divides $\prod_{i \in I}
u_i - \varepsilon(I)$ for some $I \in \mathcal D$ and $q \nmid u_1 \cdots
u_n$?
\end{question}
Note that Question \ref{quest2} is not \textit{really} a generalization of Question
\ref{quest1}, as the former can be stated in terms of the
latter by replacing, with the same notation as above, $n$ with the total number
$d$ of the prime divisors of $u_1 \cdots u_n$ and $\mathcal D$ with a
suitable subfamily of $\mathcal P_\star(S_d)$.

Questions  \ref{quest1} and \ref{quest2} are somewhat reminiscent of cyclic
systems of simultaneous congruences, studied by several authors, and still in
recent years, for their
connection with some long-standing questions in the theory of numbers, and
especially Zn\'am's problem and the Agoh-Giuga conjecture (see
\cite{Brenton02} and \cite{Laga010},
respectively, and references therein).

Our
initial motivation has been, however, of a completely different sort, and in
fact related to the
following:
\begin{problem}\label{ROMO_problem}
Let $A$ be a subset of $\mathbb P$, having at least three elements, and such
that for any $B \in \mathcal P_\star(A)$ all prime
divisors
of $\prod_{p \in B} p - 1$ belong to $A$. Then $A = \mathbb P$.
\end{problem}
This served as a problem in the 4th grade of the 2003
Romanian IMO Team Selection Test, and it appears (up to minor notational
differences) as Problem 10 in \cite[p. 53]{Vornicu03}. The solution
provided in the book (p. 62) consists of two parts. In the first one, the authors aim to show that $A$ is infinite, but their argument is seen to be at least incomplete.
Specifically, their argument is as follows (we use the notation from above):

After having proved that $2$ is in $A$, they suppose by contradiction that $A$ is a finite set of size $k$ (where $k \ge 3$) and let $p_1, \ldots, p_k$ be a numbering of $A$ such that $2 = p_1 < \cdots < p_k$.

Then, they derive from the standing assumptions on $A$ that
\begin{displaymath}
p_2^\alpha + 1 = 2^{\beta+1} p_2^\gamma + 2
\end{displaymath}
for some $\alpha, \beta,\gamma \in \mathbb{N}$. But this does
not imply $1 \equiv 2 \bmod
{p_2}$ (as is stated in the book) unless
$\gamma \ne 0$, which is nowhere proved and has no obvious reason to be true.

The problem \textit{per se} is not, however, difficult, and it was used also
for the 2004 France IMO Team Selection Test (we are not aware of any official
solution published by the organizers of the competition).

Questions somewhat similar to those above have been considered by other
authors, even though under different assumptions, and mostly
focused on the properties
of the prime factorization of particular sequences (of integers) $a_0, a_1,
\ldots$ recursively defined, e.g., by formulas of the form
$a_{n+1} = 1 + a_0
\cdots a_n$; see \cite[\S{}1.1.2]{Narkie} and the references therein for an
account (for all practical purposes, we notice here that one of the questions raised by A. A. Mullin in \cite{Mullin63} and mentioned by W. Narkiewicz on page 2 of his book has been recently answered in \cite{Book12}).

Now, we have not been able to work out a complete solution of Question
\ref{quest1}, whatever this may be.
Instead, as already remarked, we solve it in some special cases. This is in fact the content of the following theorem, which is also the main result of the paper:
\begin{theorem}\label{th:1}
Given an integer $n \ge 3$, pick distinct
primes $p_1, \ldots, p_n$, exponents $v_1,
\ldots, v_n \in
\mathbb N^+$ and a subcollection $\mathcal D$ of $\mathcal P_\star(S_n)$
such that $\mathcal D_0 \subseteq \mathcal D$, where
\begin{displaymath}
\mathcal D_0 := \mathcal
P_1(S_n) \cup \mathcal P_{n-2}(S_n) \cup \mathcal
P_{n-1}(S_n).
\end{displaymath}
Then, for
every function $\varepsilon: \mathcal P_\star(S_n) \to \{\pm 1\}$ such that
the restriction
of $\varepsilon$ to $\mathcal D_0$ is constant, there exists at least one
$q \in \mathbb P \setminus \{p_1, \ldots, p_n\}$ such that $q$ divides
$\prod_{i \in I}
p_i^{v_i} - \varepsilon(I)$ for some $I \in \mathcal D$.
\end{theorem}
The proof of Theorem \ref{th:1}, as presented in Section \ref{sec:th1}, requires a number of preliminary lemmas,
which are stated and proved under assumptions much weaker than those in the
above statement.

In particular, we will make use at some point of the following
result \cite{Zsig}:
\begin{theorem}[Zsigmondy's theorem]
\label{th:zsigmondy}
Pick $a,b \in \mathbb N^+$ and an integer $n \ge 2$ such that (i) $a > b$ and
(ii) neither $(a,b,n) = (2,1,6)$ nor $a+b$ is a power of $2$ and $n = 2$. Then,
there exists a prime $p$ such
that $p \mid a^n-b^n$ and $p \nmid a^k-b^k$ for each positive integer $k < n$.
\end{theorem}
Theorem \ref{th:1} can be used to solve a generalization of Problem
\ref{ROMO_problem}, for which we need to introduce some more notation.

Specifically, for $B,C \subseteq \mathbb Z$ we write
$B \perp C$ if for every $b
\in B$ there exists
$c \in C$ such that $b \mid c$; this simplifies to
$b \perp C$ when $B =
\{b\}$. It is clear that $B \perp C$ if and only if $b \perp C$ for
all $b \in B$.

Based on these premises, we then prove the
following:
%
\begin{theorem}\label{th:2}
Pick $\varepsilon_0 \in \{\pm 1\}$ and let $A$ be a
set of prime powers with the property that $|A| \ge 3$ and $q \perp A$ whenever $q$
is a prime dividing $\prod_{a \in B}
a - \varepsilon_0$ for some
$B \in \mathcal P_\star(A)$.
Then, $A$ is infinite. Also, $\mathbb P \perp A$
if $\varepsilon_0 = 1$. Finally, $A = \mathbb P$ if $A \subseteq \mathbb P$.
\end{theorem}
Theorem \ref{th:2} is proved in Section \ref{sec:th2}.
Incidentally, the result gives a solution of Problem
\ref{ROMO_problem} in the special case where $\varepsilon_0 = 1$ and $A
\subseteq \mathbb P$, while providing another proof, although overcomplicated,
of the infinitude of primes.

The conclusions of Theorem \ref{th:2} leads to the following:
\begin{question}
\label{quest3}
Let $\mathfrak P$ be an infinite set of primes. Does there exist a set of prime powers, say $A$, such
that $q \perp A$ for some $q \in \mathbb P$ if and only if $q$ is a prime divisor of $\prod_{a \in B} a + 1$ for some $B \in \mathcal P_\star(A)$ and $q \in \mathfrak P$? If not, what about a ``non-trivial'' characterization of those $\mathfrak P$ for which this happens?
\end{question}

Another question along the same lines is as follows:

\begin{question}
\label{quest3b}
Let $\mathfrak P$ be an infinite set of primes and pick $\varepsilon_0 \in \{\pm 1\}$. Does there
exist a set $A$ of prime powers such that $q \in \mathfrak P$ if and only if $q$ is a prime divisor of $\prod_{a \in B} a - \varepsilon_0$ for some $B \in \mathcal P_\star(A)$? If not, can we provide a ``non-trivial'' characterization of those $\mathfrak P$ for which this is true?
\end{question}

Both of these questions are almost completely open to us. Two related (but easier) questions are answered by Examples \ref{rem:2} and \ref{rem:1} in Section \ref{sec:th2}.

\section{Preparations}\label{sec:preparations}
Here below, we fix some more notation and prove a few
preliminary lemmas related to  Question \ref{quest1} in its full generality
(that is, the analysis is not restricted to the special cases covered by Theorem \ref{th:1}).

For any purpose it may serve, we recall from the introduction that, in our notation, $0 \in \mathbb N$ and $\emptyset, A \notin \mathcal P_\star(A)$ for any set $A$.

In the remainder of this section, we suppose that there exist an integer $n \ge 3$, a set $\mathfrak P = \{p_1, \ldots, p_n\}$ of $n$ primes,
integral exponents $v_1, \ldots, v_n \in \mathbb N^+$, a nonempty subfamily
$\mathcal D$ of $\mathcal P_\star(S_n)$, and a map $\varepsilon: \mathcal P_\star(S_n)
 \to \{\pm 1\}$ such that $p_1 < \cdots < p_n$ and $q \in \mathfrak P$
whenever $q \in \mathbb P$ and $q$ divides $\prod_{i \in I} p_i^{v_i} -
\varepsilon_I$ for some $I \in \mathcal D$, where $\varepsilon_I := \varepsilon(I)$ for economy of notation.

Accordingly, we show that these assumptions lead to a contradiction if $\mathcal D$
contains some distinguished
subsets of $S_n$ and the restriction of $\varepsilon$ to the subcollection of these sets, herein denoted by $\mathcal D_0$, is constant: This is especially the case when $\mathcal D_0 = \mathcal
P_1(S_n) \cup \mathcal P_{n-2}(S_n) \cup \mathcal P_{n-1}(S_n)$.

We let $P :=
\prod_{i=1}^n p_i^{v_i}$ and $\mathcal D^\op := \{S_n \setminus I: I
\in \mathcal D\}$,
and then for each $I \in \mathcal P_\star(S_n)$ we define
\begin{displaymath}
 P_I:= \prod_{i \in  I} p_i^{v_i}, \quad P_{-I}:=
P_{S_n \setminus I} \quad\text{and}\quad \varepsilon_{-I} := \varepsilon_{S_n
\setminus I}
\end{displaymath}
(notice that $P = P_I \cdot P_{-I}$).
In particular, given $i \in S_n$ we write
$P_i$ for $P_{\{i\}}$ and $P_{-i}$ for $P_{-\{i\}}$, but also
$\varepsilon_i$ for $\varepsilon_{\{i\}}$ and $\varepsilon_{-i}$ for
$\varepsilon_{-\{i\}}$.

It then follows from our assumptions that
there are maps $\alpha_1, \ldots, \alpha_n: \mathcal P_\star(S_n) \to \mathbb
N$ such that
\begin{equation}
\label{equ:id1}
P_{-I} = \varepsilon_{-I} + \prod_{i \in
I} p_i^{\alpha_{i,I}}\ \text{ for every } I \in \mathcal D^\op,
\end{equation}
where $\alpha_{i,I} := \alpha_i(I)$. Thus, if there exists
$i \in S_n$ such that $\{i\} \in \mathcal D^\op$ then
\begin{equation}
\label{equ:id2}
P_{-i} = p_i^{\alpha_i} + \varepsilon_{-i},\quad\text{with}\quad\alpha_i :=
\alpha_{i,\{i\}} \in \mathbb N^+
\end{equation}
(of course, $\alpha_i \ge 1$ since $P_{-i} - \varepsilon_{-i} \ge 2 \cdot 3 -
1$). This in turn implies that
\begin{equation}
\label{equ:id3}
P = P_{I_1} \cdot
\left(\varepsilon_{-I_1} + \prod_{i \in I_1} p_i^{\alpha_{i,I_1}}\right) = P_{I_2}
\cdot \left(\varepsilon_{-I_2} + \prod_{i \in I_2} p_i^{\alpha_{i,I_2}}\right),
\end{equation}
for all $I_1, I_2 \in \mathcal D^\op$, which specializes to:
\begin{equation}
\label{equ:id4}
 P = p_{i_1}^{v_{i_1}} \cdot \left(p_{i_1}^{\alpha_{i_1}} +
\varepsilon_{-i_1}\right) = p_{i_2}^{v_{i_2}}\cdot \left(p_{i_2}^{\alpha_{i_2}} + \varepsilon_{-i_2}\right)
\end{equation}
for all $i_1, i_2 \in S_n$ such that $\{i_1\}, \{i_2\} \in \mathcal D^\op$.

We mention in this respect that, for any fixed integer $b \ne 0$ and any finite subset $\mathcal S$ of $\mathbb P$, the diophantine equation
\begin{equation}
\label{equ:id5}
A \cdot (a^{x_1} - a^{x_2}) = B \cdot (b^{y_1} - b^{y_2})
\end{equation}
has only finitely many solutions in \textit{positive} integers $a,A,B,x_1, x_2, y_1, y_2$  for which $a$ is a prime, $\gcd(Aa,Bb) = 1$, $x_1 \ne x_2$ and all the prime factors of $AB$ belong to $\mathcal S$; see \cite{Bugeaud06} and the references therein. It follows that our equation \eqref{equ:id4} has only finitely many possible scenarios for $\varepsilon$ taking the constant value $-1$ in $\mathcal D$.

However, the methods used in \cite{Bugeaud06} are not effective and, as far as we can tell, a list of all the solutions to equation \eqref{equ:id5} is not known, not even in the special case when $A = B = 1$ and $b = 2$. Furthermore, there does not seem to be any obvious way to adapt the proof of the main result in \cite{Bugeaud06} to cover all of the cases resulting from equation \eqref{equ:id4}.

With this in mind, and based
on \eqref{equ:id1}, our main
hypothesis can be now restated as
\begin{equation}
\label{equ:hypo}
\text{``}q \mid P_{-I} - \varepsilon_{-I}\ \text{for some }q \in \mathbb P \text{ and }I \in
\mathcal D^\op \text{ only if } q \in \mathfrak P\text{''}.
\end{equation}
In addition, we can easily derive, using \eqref{equ:id3} and unique
factorization, that
\begin{equation}
\label{equ:xfactor}
\text{``}\textstyle q \mid \varepsilon_{-I} + \prod_{i \in I}
p_i^{\alpha_{i,I}}\ \text{for some }q \in \mathbb P\text{ and }I \in \mathcal D^\op \text{ only if } q \in \mathfrak
P\text{''}.
\end{equation}
Both of \eqref{equ:hypo} and \eqref{equ:xfactor} will be often referred to
throughout the article. Lastly, we say that $\varepsilon$ is
\textit{$k$-symmetric} for a certain $k \in \mathbb N^+$ if both of the following
conditions hold:
\begin{displaymath}
\text{(i) } I \in \mathcal D \cap \mathcal P_k(S_n)\text{ only if }I \in
\mathcal D^\op; \quad \text{(ii) } \varepsilon_I = \varepsilon_{-I}\text{ for all
}I \in \mathcal D \cap \mathcal P_k(S_n).
\end{displaymath}
With all this in hand, we are finally ready to prove a few preliminary results
that will be used later, in Section \ref{sec:th1}, to establish our
main theorem.
\subsection{Preliminaries}
The material is intentionally organized
into a list of lemmas based on ``local'', rather than ``global'', hypotheses.

This is motivated by the
idea of highlighting which
is used for which purpose, in the hope that this can help find an approach to solve Question
\ref{quest1} in a broader generality.

In particular, the first half of Theorem
\ref{th:1}, namely the one corresponding to the case $\varepsilon_0 = 1$, will be
an immediate corollary of Lemma \ref{lem:neven} below (the second part needs more work).

In what follows, given $a,b \in \mathbb Z$ with $a^2
+ b^2 \ne 0$ we use $\gcd(a,b)$ for the greatest common divisor of $a$ and $b$. Furthermore, for every $m \in \mathbb N^+$ such that $\gcd(a,m) = 1$ we denote by ${\rm ord}_m(a)$ the smallest $k \in
\mathbb N^+$ for which $a^k \equiv 1 \bmod m$.
\begin{lemma}
\label{lem:expos}
If $p_{i} = 3$ for some $i \in S_n$ and there exists $j \in S_n \setminus
\{i\}$ such that $\{j\} \in \mathcal D^\op$, then one, and only one, of the
following conditions holds:
\begin{enumerate}[label={\rm\arabic{*}.}]
\item $\varepsilon_{-j} = -1$ and $\alpha_j$ is even.
\item $\varepsilon_{-j} = -1$, $\alpha_j$ is odd and $p_j \equiv 1 \bmod
6$.
\item $\varepsilon_{-j} = 1$, $\alpha_j$ is odd and $p_j \equiv 2 \bmod
3$.
\end{enumerate}
\end{lemma}
\begin{proof}
The hypotheses and equation \eqref{equ:id4} give that $3 \mid
p_j^{\alpha_j}+\varepsilon_{-j}$, which is possible only if one, and only one,
of the desired conditions is satisfied.
\end{proof}
The next lemma, as elementary as it is, provides a sufficient condition under
which $2 \in \mathfrak P$. (As a rule of thumb, having a way to show that $2$ and $3$ are in
$\mathfrak P$ looks like a key aspect of the problem in its full generality.)
\begin{lemma}
\label{lem:p1=2}
If there exists $I \in \mathcal D$ such that $1 \notin I$ then
$p_1 = 2$; moreover, $\alpha_1 \ge 4$ if, in addition to the other assumptions, $I
\in \mathcal P_{n-1}(S_n)$.
\end{lemma}
\begin{proof}
Clearly, $p_i$ is odd for
each $i \in I$, which means that $P_I - \varepsilon_I$ is even, and hence $p_1 =
2$ by \eqref{equ:hypo} and the assumed ordering of the primes $p_i$. Thus, it
follows from \eqref{equ:id2} that if $I \in \mathcal P_{n-1}$ then $2^{\alpha_1}
=
P_{-1} - \varepsilon_{-1} \ge 3 \cdot 5 - 1$, with the result that
$\alpha_1 \ge 4$.
\end{proof}
The following two lemmas prove that, in the case of a $1$-symmetric
$\varepsilon$, mild hypotheses imply that $3 \in \mathfrak P$.
\begin{lemma}
\label{lem:p=1modq}
Suppose that $\varepsilon$ is $1$-symmetric and pick a prime $q \notin \mathfrak
P$. Then, there does not exist any $i \in S_n$ such that $\{i\} \in \mathcal D$
and $p_i \equiv 1 \bmod q$.
\end{lemma}
\begin{proof}
Assume for the sake of contradiction that there exists $i_0 \in S_n$ such that
$\{i_0\} \in \mathcal D$ and $p_{i_0} \equiv 1 \bmod q$.
Then, using that $\varepsilon$ is $1$-symmetric, we get from
\eqref{equ:id1} and \eqref{equ:id2} that
\begin{displaymath}
 1 - \varepsilon_0 \equiv p_{i_0}^{v_{i_0}} - \varepsilon_0 \equiv
\prod_{i \in I_0} p_i^{\alpha_{i,I_0}}\bmod q
\end{displaymath}
and
\begin{displaymath}
 P_{I_0}
\equiv p_{i_0}^{\alpha_{i_0}} + \varepsilon_0 \equiv 1 + \varepsilon_0 \bmod q,
\end{displaymath}
where $I_0 := S_n \setminus \{i_0\}$. But $q \notin \mathfrak P$ implies $q \nmid p_{i_0}^{v_{i_0}}-\varepsilon_0$ by \eqref{equ:hypo}, with the result that $\varepsilon_0=-1$ (from the above), and then $q \mid P_{I_0}$.

By unique factorization, this is however in contradiction to the fact that $q$ is not in $\mathfrak P$.
\end{proof}
\begin{lemma}
\label{lem:p2=3}
Let $\varepsilon$ be $1$-symmetric and suppose there exists $J \in \mathcal
P_\star(S_n)$ such that $|S_n \setminus J|$ is even,
$\mathcal D_0 := \mathcal P_1(S_n) \cup \{S_n \setminus J\} \subseteq
\mathcal D$, and the restriction of $\varepsilon$ to $\mathcal D_0$ is
constant. Then $p_2 = 3$ and $\alpha_2 \ge \frac{1}{2}(5-\varepsilon_0)$.
\end{lemma}
\begin{proof}
Let $\varepsilon$ take the constant value $\varepsilon_0$ when restricted to
$\mathcal D_0$ and assume by contradiction that $3 \notin \mathfrak P$.

Then,
Lemma \ref{lem:p=1modq} gives that $p_i \equiv - 1 \bmod 3$ for all $i \in
S_n$, while taking $I = S_n \setminus \{i\}$ in \eqref{equ:id1} and working
modulo $3$ entails by \eqref{equ:hypo} that
\begin{displaymath}
 p_i^{v_i} - \varepsilon_0 \equiv \prod_{j \in I} p_j^{\alpha_{j,I}}
\not\equiv 0 \bmod 3,
\end{displaymath}
so that $v_i$ is odd if $\varepsilon_0 = 1$ and even otherwise (here is where we use that $\mathcal P_1(S_n) \in \mathcal D$ and $\varepsilon$ is
$1$-symmetric, in such a way that $\mathcal P_{n-1}(S_n) \in \mathcal D$ too).
Now, since $S_n \setminus J \in \mathcal D$, the same kind of reasoning
also yields that
\begin{displaymath}
 1 - \varepsilon_0 \equiv P_{-J} - \varepsilon_0 \equiv \prod_{j \in
J} p_j^{\alpha_{j,J}} \bmod 3,
\end{displaymath}
with the result that if
$\varepsilon_0 = 1$ then $3 \in \mathfrak P$ by \eqref{equ:hypo}, as follows from the fact that $S_n
\setminus J$ has an even number of elements and $v_i$ is odd for each $i \in J$ (which was proved before).
This is however a contradiction.

So we are left with the case $\varepsilon_0 = -1$. Since $-1$ is not a
quadratic residue modulo a prime $p \equiv -1 \bmod 4$, we get from the above and
\eqref{equ:id2} that in this case $p_i \equiv 1 \bmod 4$ for each $i = 2, \ldots, n$.

Therefore, \eqref{equ:id1} together with Lemma \ref{lem:p1=2} gives that $P_{-1} + 1
= 2^{\alpha_1}$ with $\alpha_1 \ge 2$, which is again a contradiction as it
means that $2 \equiv 0 \bmod 4$.

All of this proves that $p_2 = 3$, which in turn implies
from \eqref{equ:id2}  that $3^{\alpha_2} = P_{-2} - \varepsilon_{-2} \ge 2 \cdot 5 - \varepsilon_0$ (since $\varepsilon$ is $1$-symmetric and its restriction to $\mathcal D_0$ is constantly equal to $\varepsilon_0$, we have $\varepsilon_{-2} = \varepsilon_0$), so $\alpha_2 \ge 2$ if $\varepsilon_0 = 1$ and $\alpha_2 \ge 3$ if $\varepsilon_0 = -1$, i.e., $\alpha_2 \ge \frac{1}{2}(5-\varepsilon_0)$ in both cases.
\end{proof}
We now show that, if $\mathcal{D}$ contains some distinguished
subsets of $S_n$ and $\varepsilon$ is subjected to certain conditions,
then $p_i$ must be a Fermat prime.
\begin{lemma}\label{lem:fermat}
Let $\mathcal P_1(S_n
\setminus \{1\}) \subseteq \mathcal D^\op$ and assume there exists $i \in
S_n \setminus \{1\}$ such that $\{i\} \in \mathcal D$ and $\varepsilon_{\pm i} = 1$. Then, $p_i$
is a Fermat prime.
\end{lemma}
\begin{proof}
It is clear from Lemma \ref{lem:p1=2} that $p_1 = 2$.
Suppose by contradiction that there exists an odd prime $q$ such that $q \mid
p_i - 1$ (note that $p_i \ge 3$), and hence $q \mid p_i^{v_i} - \varepsilon_i$.

Then, taking $I = \{i\}$ in \eqref{equ:hypo} gives that $q = p_j$ for some $j
\in S_n \setminus \{1,i\}$. Considering that $\mathcal P_1(S_n \setminus \{1\}) \subseteq \mathcal D^\op$, it
follows from \eqref{equ:id4} that
\begin{displaymath}
p_j^{v_j} (p_j^{\alpha_j} + \varepsilon_{-j}) =
p_i^{v_i} (p_i^{\alpha_i} + 1),
\end{displaymath}
where we use that $\varepsilon_{-i} = 1$. This is however a contradiction, because
it implies that $0 \equiv 2 \bmod p_j$ (with $p_j \ge 3$). So, $p_i$ is a Fermat
prime by \cite[Theorem 17]{Hardy08}.
\end{proof}
\begin{lemma}
\label{lem:neven}
Let $\mathcal P_1(S_n) \subseteq
\mathcal D^\op$ and suppose that $p_i = 3$ for some $i \in S_n$ and there exists $j \in S_n \setminus \{1,i\}$ such that $\{j\}
\in \mathcal D$ and $\varepsilon_{\pm j} = 1$. Then $i = 2$, $p_1 = 2$, and
$\varepsilon_{-1}=-1$.
\end{lemma}
\begin{proof}
First, we have by Lemma \ref{lem:p1=2} that $p_1 = 2$, and hence $i = 2$.
Also, $p_j$ is a Fermat prime by Lemma \ref{lem:fermat} (and
clearly $p_j \ge 5$). So assume for a contradiction that
$\varepsilon_{-1}=1$.

Then, Lemma \ref{lem:expos} and \eqref{equ:id2} imply that
$p_j \mid P_{-1} = 2^{\alpha_1}+1$ with $\alpha_1$ odd, with the result that
$2 \le {\rm ord}_{p_j}(2) \le \gcd(2\alpha,
p_j - 1) = 2$. It follows that $5 \le p_j \le 2^2-1$, which is obviously
impossible.
\end{proof}
The proof of the next lemma depends on Zsigmondy's theorem. Although not
strictly related to the statement and the assumptions of Theorem \ref{th:1}, it
will be of great importance later on.
\begin{lemma}
\label{lem:sort_of_diophantine}
Pick $p,q \in \mathbb P$ and assume that there exist
$x,y,z \in \mathbb N$ for which $x \ne 0$, $y \ge 2$,
$p \mid q+1$ and $q^x - 1 = p^y(q^z - 1)$. Then $x=2$,
$z=1$, $p=2$, $y \in \mathbb P$, and $q=2^y-1$.
\end{lemma}
\begin{proof}
Since $x \ne 0$, it is clear that $q^x - 1 \ne 0$, with the
result that $z \ne 0$ and $q^z - 1 \ne 0$ too. Therefore,
using also that $y \ne 0$, one has that
\begin{equation}
\label{equ:zsigmondyeq}
p^y = (q^x - 1)/(q^z-1) > 1,
\end{equation}
which is obviously possible only if
\begin{equation}
\label{equ:trivial_ineq}
x > z \ge 1.
\end{equation}
We claim that $x \le 2$. For suppose to the contrary that $x > 2$. Then by
Zsigmondy's theorem, there must exist at least one $r \in \mathbb P$ such that
$r \mid q^x - 1$ and
\begin{displaymath}
r \nmid q^k - 1\text{ for each positive integer }k < x.
\end{displaymath}
In particular, \eqref{equ:zsigmondyeq} yields that $r = p$ (by unique
factorization), which is a contradiction since $p \mid q^2 - 1$. Thus, we get
from \eqref{equ:trivial_ineq} that $x = 2$ and $z = 1$. Then, $p^y = q+1$, that
is $p^y - 1 \in \mathbb P$, and this is absurd unless $p = 2$ and $y \in \mathbb
P$. The claim follows.
\end{proof}
This completes the preliminaries, and we can now proceed to the
proof of the main result of the paper.
\section{Proof of Theorem \ref{th:1}}
\label{sec:th1}
Throughout we use the same notation and
assumptions as in Section \ref{sec:preparations}, but we specialize to the case
where
\begin{displaymath}
\mathcal D_0 := \mathcal P_1(S_n) \cup \mathcal P_{n-2}(S_n) \cup
\mathcal P_{n-1}(S_n)
\subseteq \mathcal D
\end{displaymath}
and $\varepsilon$ takes the constant
value $\varepsilon_0$ when restricted to $\mathcal D_0$ (as in the statement of
Theorem \ref{th:1}).
\begin{proof}[Proof of Theorem \ref{th:1}]
At least one of $n-2$ or $n-1$ is even, so we have by Lemmas \ref{lem:p1=2} and
\ref{lem:p2=3} that $p_1 = 2$,
$p_2 = 3$ and $v_2 \ge 2$.

There is, in consequence, no loss of
generality in assuming, as we do, that
$\varepsilon_0  = -1$, since the other case is impossible by Lemma
\ref{lem:neven}.

Thus,
pick $i_0 \in S_n$ such that $3 \mid p_{i_0} + 1$. It follows from \eqref{equ:id3}
and our hypotheses that there exist $\beta_{i_0}, \gamma_{i_0} \in \mathbb N$ such that
\begin{displaymath}
P = 3^{v_2}(3^{\alpha_2} - 1) = p_{i_0}^{v_{i_0}}\cdot \left(p_{i_0}^{
\alpha_{i_0}}-1\right)=3^{v_2}
p_{i_0}^{v_{i_0}}\cdot \left(3^{\beta_{i_0}} p_{i_0}^{\gamma_{i_0}}-1\right),
\end{displaymath}
with the result that, on the one hand,
\begin{equation}
\label{equ:dioph1}
p_{i_0}^{\alpha_{i_0}} - 1 = 3^{v_2} \cdot \left(3^{\beta_{i_0}}
p_{i_0}^{\gamma_{i_0}} -1\right),
\end{equation}
and on the other hand,
\begin{equation}
\label{equ:dioph2}
3^{\alpha_2} - 1 = p_{i_0}^{v_{i_0}} \cdot \left(3^{\beta_{i_0}}
p_{i_0}^{\gamma_{i_0}} -1\right).
\end{equation}
Then, since $v_2 \ge 2$ and $\alpha_{i_0} \ne 0$, we see by \eqref{equ:dioph1}
and Lemma  \ref{lem:sort_of_diophantine} that $\beta_{i_0} \ge 1$. It is then
found from \eqref{equ:dioph2} that $-1 \equiv (-1)^{v_{i_0} + 1} \bmod 3$, i.e.,
$v_{i_0}$ is even. To wit, we have proved that
\begin{equation}
\label{equ:final}
\forall i \in S_n: p_i \equiv -1 \bmod 3\implies v_i\text{ is even and }
p_i^{v_i}  \equiv 1 \bmod 3.
\end{equation}
But every prime $\ne 3$ is congruent to $\pm 1$ modulo $3$. Therefore, we get from
\eqref{equ:id2} and \eqref{equ:final} that
\begin{displaymath}
 2 \equiv \prod_{i\in S_n \setminus \{2\}}{p_i^{v_i}} + 1 \equiv 3^{\alpha_2}
\equiv 0 \bmod 3,
\end{displaymath}
which is obviously a contradiction and completes the proof.
\end{proof}
\section{Proof of Theorem \ref{th:2}}
\label{sec:th2}
In the present section, unless differently specified, we use the same notation and
assumptions of Theorem \ref{th:2}, whose proof is split into three lemmas (one for each aspect of the claim).
\begin{lemma}
\label{lem:oo}
$A$ is an infinite set.
\end{lemma}
\begin{proof}
Suppose to a contradiction that $A$ is finite and let $n := |A|$.

Since $A$ is a
set of prime powers, there then exist $p_1, \ldots, p_n \in
\mathbb P$ and $v_1, \ldots, v_n \in \mathbb N^+$ such that $p_1 \le \cdots \le p_n$ and $A = \{p_1^{v_1}, \ldots, p_n^{v_n}\}$, and
our assumptions give that
\begin{equation}
\label{equ:riduzione}
\text{``}q \text{ divides }\textstyle \prod_{i \in I}
p_i^{v_i} -  \varepsilon_0\text{ for some }I \in \mathcal P_\star(S_n)\text{ only if } q \in \mathfrak P\text{''},
\end{equation}
where $\mathfrak P := \{p_1, \ldots,
p_n\}$ for brevity's sake.

This clearly implies that $p_1 < \cdots < p_n$. In fact, if $p_{i_1} =
p_{i_2}$ for distinct $i_1,i_2 \in S_n$, then it is found from
\eqref{equ:riduzione} and unique factorization that
\begin{displaymath}
 p_{i_1}^k = \prod_{i \in S_n \setminus \{i_1\}} p_i^{v_i} - \varepsilon_0
\end{displaymath}
for a certain $k \in \mathbb N^+$, which is impossible when reduced modulo
$p_{i_1}$.

Thus, using that $n \ge 3$, it follows from Theorem \ref{th:1} that
there also exists $q \in \mathbb P \setminus \mathfrak P$ such that $q$ divides
$\prod_{i \in I} p_i^{v_i} - \varepsilon_0$ for some $I \in \mathcal
P_\star(S_n)$. This is, however, in contradiction to \eqref{equ:riduzione},
and the proof is complete.
\end{proof}
\begin{lemma}\label{lem:eps0=1final}
If $\varepsilon_0 = 1$, then $\mathbb P \perp A$. In particular, $A = \mathbb P$
if $A \subseteq \mathbb P$.
\end{lemma}
\begin{proof}
Suppose for the sake of contradiction that there exists $p \in \mathbb P$ such
that $p$ does not divide any element of $A$.

Since $|A| = \infty$ (by Lemma \ref{lem:oo}),
this together with the pigeonhole
principle implies that, for
a certain $r \in S_{p-1}$, the set
$$
A_{r} := \{a \in A: a \equiv r \bmod p\}
$$
is infinite, and we have that
\begin{equation}
\label{equ:fermatl}
 \forall B \in \mathcal P_\star(A_{r}): \prod_{a \in
B} a \equiv  \prod_{a \in B} r \equiv r^{|B|} \bmod p.
\end{equation}
As it is now possible to choose $B_0 \in \mathcal P_\star(A_{r})$ in such a way
that $|B_0|$ is a multiple of $p-1$, one gets from \eqref{equ:fermatl} and
Fermat's little theorem
that
$p$ divides a number of the form
$\prod_{a \in B} a - 1$ for some $B \in \mathcal P_\star(A)$, and hence $p \mid
a_0$ for some $a_0 \in A$ (by the assumptions of Theorem \ref{th:2}).

This is, however, absurd, because by construction
no element of $A$ is divisible by $p$. It follows that $\mathbb P \perp A$, and the rest is
trivial.
\end{proof}
In the next lemma,
we let an empty sum be equal to $0$ and an empty product be equal to $1$,
as usual.
\begin{lemma}
\label{lem:eps0=-1final}
If $\varepsilon_0 = -1$ and $A \subseteq \mathbb P$, then $A = \mathbb P$.
\end{lemma}
\begin{proof}
Suppose to a contradiction that there exists $p \in
\mathbb{P}$ such that $p \notin A$, and for each $r \in S_{p-1}$ let $A_r := \{a
\in A: a \equiv r \bmod p\}$. Then,
\begin{equation}
\label{equ:unionAi}
A = A_1 \cup \cdots \cup A_{p-1}.
\end{equation}
In addition to this, set $
\Gamma_{\rm fin} := \{r \in S_{p-1}: |A_r| < \infty \}$ and $\Gamma_{\rm inf}
:= S_{p-1} \setminus \Gamma_{\rm fin}$, and take
\begin{displaymath}
A_{\rm fin} := \bigcup_{r \in \Gamma_{\rm fin}} A_r
\quad\text{and}\quad A_{\rm inf} := A \setminus A_{\rm fin}.
\end{displaymath}
It is clear from \eqref{equ:unionAi} that $A_{\rm inf}$ is infinite, because
$A_{\rm fin}$ is finite, $\{A_{\rm fin}, A_{\rm inf}\}$ is a partition of $A$,
and $|A| = \infty$ by Lemma \ref{lem:oo}. So, we let $\xi_0 := \prod_{a \in
A_{\rm fin}} a$.

We claim that there exists a sequence $\varrho_0, \varrho_1, \ldots$ of
positive integers such that $\varrho_n$ is, for each $n \in \mathbb N$,  a
nonempty product (of a finite number) of distinct elements of $A$
with the property that
\begin{equation}
\label{equ:recursion}
 \xi_0 \mid \varrho_n \quad \text{and}\quad 1 + \varrho_n \equiv
\sum_{i=0}^{n+1} \varrho_0^i \bmod p.
\end{equation}
\begin{proof}[Proof of the claim] We construct the sequence $\varrho_0,
\varrho_1, \ldots$ in a recursive way. To start with, pick an arbitrary $a_0 \in
A_{\rm inf}$ and define $\varrho_0 := a_0 \xi_0$,
where the factor $a_0$ accounts for the possibility that $\Gamma_{\rm fin} =
\emptyset$.

By construction, $\varrho_0$ is a nonempty product of
distinct elements of $A$, and \eqref{equ:recursion} is satisfied in the base case $n = 0$.

Now fix $n \in \mathbb N$ and suppose that we have already found
$\varrho_n \in \mathbb N^+$ such that $\varrho_n$ is a product of
distinct elements of $A$ and \eqref{equ:recursion} holds true with $\varrho_0$ and $\varrho_n$. By unique factorization,
there then exist exponents $s_1,
\ldots, s_k \in \mathbb N^+$ and distinct primes
$p_1, \ldots, p_k \in \mathbb P$ ($k \in \mathbb N^+$) such that
\begin{equation}
\label{equ:factor_rho0}
 \xi_0 \mid  \varrho_n \quad \text{and}
\quad 1 + \varrho_n = \prod_{i=1}^k p_i^{s_i}.
\end{equation}
Therefore, we get from the assumptions on $A$ that $p_i \perp A$ for each $i \in S_k$, which in turn implies that $p_i \in A$ (since $A \subseteq \mathbb P$ by hypothesis), and actually $p_i
\in A_{\rm inf}$, considering that every element of $A_{\rm fin}$, if any exists, is a
divisor of $\xi_0$, and $\xi_0 \mid \varrho_n$ by
\eqref{equ:factor_rho0}.

Using that $A_r$ is infinite for every $r \in \Gamma_{\rm inf}$ and $A_{\rm inf}
= \bigcup_{r \in \Gamma_{\rm inf}} A_r$, this yields that there exist
$a_1, \ldots, a_h \in A_{{\rm inf}}$ such that,
on the one hand,
\begin{equation}
\label{equ:distinct_all}
\varrho_0 < a_1 < \cdots < a_h,
\end{equation}
and on the other hand,
\begin{equation}
\label{equ:long_chains}
p_i \equiv a_{1+t_i} \equiv  \cdots
\equiv a_{s_i+t_i} \bmod p
\end{equation}
for every $i \in S_k$, where we define $h := \sum_{i=1}^k s_i$ and $t_i := \sum_{j=1}^{i-1}
s_j$. It follows from \eqref{equ:factor_rho0} and
\eqref{equ:long_chains} that
\begin{displaymath}
 1 + \varrho_n \equiv \prod_{i=1}^k p_i^{s_i} \equiv
\prod_{i=1}^h a_i \bmod p.
\end{displaymath}
So, by the assumptions on $\varrho_n$ and the above considerations, we
see that
\begin{displaymath}
1 + \varrho_0 \cdot \prod_{i=1}^h a_i \equiv 1 + \varrho_0 \cdot (1 + \varrho_n) \equiv 1 + \varrho_0
\cdot \sum_{i=0}^{n+1} \varrho_0^{i} \equiv \sum_{i=0}^{n+2}
\varrho_0^{i} \bmod p.
\end{displaymath}
Our claim is hence proved (by induction) by taking $\varrho_{n+1} :=
\varrho_0 \cdot \prod_{i=1}^h a_i$, because $\xi_0 \mid \varrho_0 \mid \varrho_{n+1}$ and $\varrho_{n+1}$ is, by virtue of
\eqref{equ:distinct_all}, a nonempty product of distinct elements of $A$.
\end{proof}
Thus, letting $n = p(p-1)-2$ in \eqref{equ:recursion} and observing that $p
\nmid \varrho_0$ (since $p\notin A$ and $\varrho_0$ is, by construction, a product
of elements of $A$) give that $1 + \varrho_n \equiv 0 \bmod p$, with the result
that $p \in A$ by the assumed properties of $A$. This is, however, a
contradiction, and the proof is complete.
\end{proof}

Finally, we have all the ingredients for the following:

\begin{proof}[Proof of Theorem \ref{th:2}]
Just put together Lemmas \ref{lem:oo}, \ref{lem:eps0=1final} and \ref{lem:eps0=-1final}.
\end{proof}

We conclude the section with a couple of examples, the first of which provides evidence of a substantial difference between Lemmas \ref{lem:eps0=1final} and
\ref{lem:eps0=-1final}, and is potentially of interest in relation to Question
\ref{quest3}.
\begin{example}
\label{rem:2}
Given $\ell \in \mathbb N^+$ and odd primes $q_1, \ldots, q_\ell$, let
$$
k := {\rm lcm}(q_1 - 1, \ldots, q_\ell - 1)
$$
and
$$
A := \{p^{nk}: p \in \mathbb P \setminus \mathcal Q, n \in \mathbb N^+\},
$$
where $\mathcal Q := \{q_1, \ldots, q_\ell\}$.
We denote by $\mathfrak P$ the set of all
primes $q$ for which there exists $B \in \mathcal P_\star(A)$ such that $q$
divides $\prod_{a \in B} a + 1$.

It is then easily seen that $\mathfrak P
\subseteq \mathbb P \setminus \mathcal Q$, since for every $B \in \mathcal P_\star(A)$ and each $i = 1, \ldots, \ell$ Fermat's little theorem gives
$\prod_{a \in B} a + 1 \equiv 2 \not\equiv 0 \bmod q_i$.

On the other hand, the very definition of $A$ yields that $q \perp A$, for some $q \in \mathbb P$, if and only if $q \notin \mathcal Q$.
\end{example}

The example above shows that, given a finite nonempty $\mathcal Q \subseteq \mathbb P$, there exists a set $A$ of prime powers such that the set of primes dividing at least one number of the form $\prod_{a \in B} a + 1$ for some $B \in \mathcal P_\star(A)$ is contained in $\mathbb P \setminus \mathcal Q$, and Question \ref{quest3} asks if this inclusion can be actually made into an equality for a suitable $A$.

The next example, on the other hand, may be of interest in relation to Question \ref{quest3b}.
\begin{example}\label{rem:1}
For $\ell \in \mathbb N^+$ pick distinct primes $q_1, \ldots, q_\ell \ge 3$
and, in view of \cite[Theorem 110]{Hardy08}, let $g_i$ be a primitive root modulo $q_i$.

A standard argument based on the Chinese remainder theorem shows that there also exists an integer $g$
such that $g$ is a primitive root modulo $q_i$ for each $i$, and by Dirichlet's theorem on arithmetic progressions we can choose
$g$ to be prime. Now, fix $\varepsilon_0 \in \{\pm 1\}$ and define
\begin{displaymath}
A :=
\left\{
\begin{array}{ll}
\bigcup_{i=1}^\ell \{g^{(q_i-1)n}: n \in \mathbb N^+\} & \text{if }\varepsilon_0
= 1 \\
   & \\
\bigcup_{i=1}^\ell \{g^{\frac{1}{2}(q_i-1)(2n - 1)}: n \in \mathbb N^+\} &
\text{if }\varepsilon_0 = -1
\end{array}
\right..
\end{displaymath}
If $\mathfrak P$ is the set of all primes $q$ such that $q$ divides $\prod_{a
\in B} a - \varepsilon_0$ for some $B \in \mathcal P_\star(A)$, then on the one
hand, $q_i \in \mathfrak P$ for each $i$ (essentially by construction),
and on the other hand, no element of $A$ is divided by $q_i$ (because $g$ and $q_i$ are coprime).
\end{example}

\section{Closing remarks}\label{sec:final}
Many ``natural" questions related to the ones already stated in
the previous sections arise, and perhaps it can be interesting
to find them an answer.

Here are some examples: Is it possible to prove Theorem \ref{th:1} under the
weaker assumption that $\mathcal D_0$, as there defined, is
$\mathcal P_1(S_n) \cup \mathcal P_{n-1}(S_n)$ instead of
$\mathcal P_1(S_n) \cup \mathcal P_{n-2}(S_n) \cup \mathcal
P_{n-1}(S_n)$? This is clearly the case if $n = 3$, but what
about $n \ge 4$? And what if $n$ is sufficiently large and
$\mathcal D_0 = \mathcal P_k(S_n)$ for some $k \in S_n$? The
answer to the latter question is negative for $k = 1$ (for, take $p_1, \ldots, p_n$
to be the $n$ smallest primes and let $v_1 = \cdots = v_n =
\varepsilon_0 = 1$, then observe that, for each $i \in S_n$, the
greatest prime divisor of $p_i^{v_i} - \varepsilon_0$ is
$\le p_i - 1$). But what if $k \ge 2$?

In addition to the above: To what
degree can the results of Section \ref{sec:preparations} be
extended in the direction of Question \ref{quest2}? It seems worth mentioning in this respect that Question \ref{quest2} has
the following ``abstract'' formulation  (we refer to \cite[Ch. 1]{Mollin11} for background on divisibility and
related topics in the general theory of rings):
\begin{question}
\label{quest4}
Given an integral domain $\mathbb F$ and an integer $n \ge 3$, pick
pairwise coprime non-units $u_1, \ldots, u_n \in \mathbb F$ (assuming that
this is actually possible), and let $\mathcal D$ be a nonempty subfamily of
$\mathcal P_\star(S_n)$ with ``enough'' elements. Does there exist at least one
irreducible $q \in \mathbb F$ such that $q$ divides $\prod_{i \in I}
u_i - 1$ for some $I \in \mathcal D$ and $q \nmid u_1 \cdots u_n$?
\end{question}
In the above, the condition that $u_1, \ldots, u_n$ are non-units is necessary
to ensure that $\prod_{i \in I} u_i -
1 \ne 0$ for each $I \in \mathcal D$ (otherwise the question would be, in a certain sense, trivial).

In fact, one may want to assume that $\mathbb F$ is a UFD, in such a way that
an
element is irreducible if and only if it is prime \cite[Theorems 1.1 and
1.2]{Mollin11}. In particular, it seems interesting to try to answer Question
\ref{quest4} in the special case where $\mathbb F$ is the ring of
integers of a quadratic extension of $\mathbb Q$ with the property of unique
factorization, and $u_1, \ldots, u_n$ are primes in $\mathbb F$. Hopefully, this will be the subject of future work.

\section*{Acknowledgments}
We are grateful to Carlo Pagano
(Universit\`a di Roma Tor Vergata) for having suggested the key idea
used in the proof of Lemma \ref{lem:eps0=-1final},
to Alain Plagne (\'Ecole Polytechnique) for
remarks that improved the readability of the paper, and to anonymous referees for helpful comments.

The second named author was supported, during the preparation of the manuscript, by the
European Community's 7th Framework Programme (FP7/2007-2013) under Grant
Agreement No. 276487 (project ApProCEM), and partly from the ANR Project No. ANR-12-BS01-0011 (project CAESAR).


\begin{thebibliography}{99}
\bibitem{AZ10} {\sc M. Aigner and G. M. Ziegler}, \textit{Proofs from THE BOOK}. 4th ed.,
Springer, 2010.
%
\bibitem{Vornicu03} {\sc M. Becheanu, M. Andronache, M. B\u{a}lun\u{a}, R. Gologan,
D. \c{S}erb\u{a}nescu, and V. Vornicu}, \textit{Romanian Mathematical Competitions
2003}. Societatea de \c{S}tiin\c{t}e Matematice din Rom\^{a}nia, 2003.
%
\bibitem{Book12} {\sc A. R. Booker}, \textit{On Mullin's second sequence of primes}. Integers {\bf A12} (2012), \#A4.
%
\bibitem{Borw96} {\sc D. Borwein, J. M. Borwein, P. B. Borwein, and R. Girgensohn},
\textit{Giuga's conjecture on primality}. Amer. Math. Monthly {\bf103} (1996),
40--50.
%
\bibitem{Brenton02} {\sc L. Brenton and A. Vasiliu}, \textit{Zn\'am's problem}.
Math. Mag. {\bf75}, No.~1 (2002), 3--11.
%
\bibitem{Bugeaud06} {\sc Y. Bugeaud and F. Luca}, \textit{On Pillai's Diophantine equation}. New York J. Math. {\bf12} (2006), 193--217.
%
%
\bibitem{Hardy08} {\sc G. H. Hardy and E. M. Wright}, \textit{An Introduction to the
Theory of Numbers}. 6th ed. (revised by D.R. Heath-Brown and J.H. Silverman), Oxford University Press, 2008.
%
\bibitem{Laga010} {\sc J. C. Lagarias}, \textit{Cyclic systems of simultaneous
congruences}. Int. J. Number Theory {\bf6}, No.~2 (2010),
219--245.
%
\bibitem{Luca03} {\sc F. Luca}, \textit{On the diophantine equation $p^{x_1} - p^{x_2} = q^{y_1} - q^{y_2}$}. Indag. Mathem. (N.S.) {\bf14}, No.~2 (2003), 207--222.
%
%
\bibitem{Mollin11} {\sc R. A. Mollin}, \textit{Algebraic Number Theory}. Discrete
Mathematics and Its Applications, 2nd ed., Chapman and Hall/CRC, 2011.
%
\bibitem{Mullin63} {\sc A. A. Mullin}, \textit{Recursive function theory (a modern look at a Euclidean idea)}.
Bull. Amer. Math. Soc. {\bf69} (1963), 737.
%
\bibitem{Narkie} {\sc W. Narkiewicz}, \textit{The Development of Prime Number
Theory}. Springer-Verlag, 2000.
%
\bibitem{Zsig} {\sc K. Zsigmondy}, \textit{Zur Theorie der Potenzreste}. Monatsh. Math. {\bf3}, No.~1 (1892), 265--284.
\end{thebibliography}
\end{document}